\documentclass{birkjour}
\usepackage{amsmath, amsthm, amscd, amsfonts,url, amssymb,graphicx,graphics, color}
\usepackage[bookmarksnumbered, plainpages]{hyperref}
 \newtheorem{thm}{Theorem}[section]
 
 \newtheorem{lem}[thm]{Lemma}
 
 \theoremstyle{definition}
 
 \theoremstyle{remark}

 \numberwithin{equation}{section}
\usepackage{pgf,tikz}
\usetikzlibrary{arrows}

\begin{document}

\title[A note on the automorphism group of the Hamming graph]
 {A note on the automorphism group of \\the Hamming graph}

\author{S.Morteza Mirafzal and Meysam Ziaee}

\address{Department of Mathematics\\ Lorestan University\\ Khoramabad\\ Iran}

\email{smortezamirafzal@yahoo.com}
\email{mirafzal.m@lu.ac.ir}
\email{masimeysam@gmail.com}
\thanks{}


\date{}

\begin{abstract}
Let $\Omega$  be a $m$-set,  where $m>1$, is an integer. The  Hamming   graph $H(n,m)$, has  $\Omega ^{n}$ as its vertex-set,  with two vertices are adjacent if and only if    they differ in exactly one coordinate.    In this paper, we provide a proof on  the automorphism group of the Hamming graph $H(n,m)$, by using elementary facts of group theory and graph theory.
\end{abstract}

\maketitle


\section{Introduction}
\noindent

Let $\Omega$  be a $m$-set,  where $m>1$, is an integer. The $Hamming  \ graph$ $H(n,m)$, has  $\Omega ^{n}$ as its vertex-set,  with two vertices are adjacent if and only if    they differ in exactly one coordinate.
This graph is very famous and much
is known about it, for instance this graph is actually the Cartesian product of $n$  complete graphs
$K_m$, that is,  $K_m \Box \cdots \Box K_m$.
 In general,
the connection between Hamming graphs and coding theory is of major importance.
If $m=2$, then $H(n,m)=Q_n$, where $Q_n$ is the hypercube of dimension $n$. Since,  the automorphism group of the hypercube $Q_n$ has been already determined [10], in the sequel, we assume that $m \geq 3$. \
Figure 1. displays $H(2,3)$ in the plane.  Note that in this figure, we denote the vertex $(x,y)$ by $xy$.\

It follows from the definition of the Hamming graph $ H(n,m) $ that if $\theta \in \mbox{Sym}([n]$), where $ \Omega=[n]=\{1,\cdots,  n  \} $,   then
$$ f_\theta : V( H(n,m) )\longrightarrow V( H(n,m)),
 f_\theta (x_1, ..., x_n ) =   (x_{\theta (1)}, ..., x_{\theta (n)} ),  $$
 is an automorphism of the Hamming graph $ H(n,m),  $ and the mapping  \
  $ \psi : \mbox{Sym} ([n]) \longrightarrow Aut ( H(n,m) )$, defined by
 this  rule,   $ \psi ( \theta ) = f_\theta $,  is an injection. Therefore, the set  $H= \{ f_\theta \ |\  \theta \in \mbox{Sym}([n]) \} $,  is a subgroup of $ Aut (H(n,m)) $,  which is isomorphic with $\mbox{Sym}([n])$.  Hence,  we have  $\mbox{Sym}([n]) \leq Aut ( H(n,m) )$.\

\definecolor{qqqqff}{rgb}{0.,0.,1.}
\begin{tikzpicture}[line cap=round,line join=round,>=triangle 45,x=.65cm,y=.8cm]
\clip(-4.3,-2.44) rectangle (11.32,6.3);
\draw (0.64,4.)-- (2.92,3.22);
\draw (2.92,3.22)-- (5.38,3.96);
\draw (5.38,3.96)-- (6.26,2.34);
\draw (6.26,2.34)-- (5.54,-0.14);
\draw (5.54,-0.14)-- (3.02,0.62);
\draw (3.02,0.62)-- (0.76,-0.28);
\draw (0.76,-0.28)-- (-0.28,2.18);
\draw (-0.28,2.18)-- (0.64,4.);
\draw (0.64,4.)-- (5.38,3.96);
\draw (5.38,3.96)-- (5.54,-0.14);
\draw (5.54,-0.14)-- (0.76,-0.28);
\draw (0.76,-0.28)-- (0.64,4.);
\draw (2.92,3.22)-- (2.4,1.64);
\draw (2.4,1.64)-- (3.02,0.62);
\draw (-0.28,2.18)-- (6.26,2.34);
\draw (2.92,3.22)-- (3.02,0.62);
\draw (-0.28,2.18)-- (2.4,1.64);
\draw (2.4,1.64)-- (6.26,2.34);
\draw (0.66,-0.38) node[anchor=north west] {00};
\draw (3.26,1.12) node[anchor=north west] {01};
\draw (5.64,-0.24) node[anchor=north west] {02};
\draw (-1.1,2.36) node[anchor=north west] {10};
\draw (1.98,1.4) node[anchor=north west] {11};
\draw (6.4,2.84) node[anchor=north west] {12};
\draw (-0.22,4.32) node[anchor=north west] {20};
\draw (2.48,3.98) node[anchor=north west] {21};
\draw (5.68,4.24) node[anchor=north west] {22};
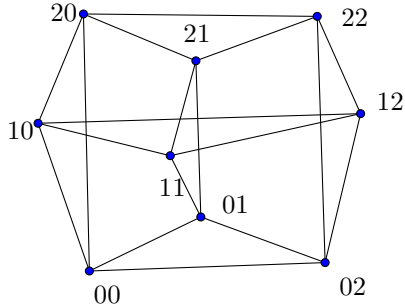
\draw (-1.1,-1.12) node[anchor=north west] {Figure 1. The Hamming graph H(2,3)};
\begin{scriptsize}
\draw [fill=qqqqff] (0.64,4.) circle (1.5pt);
\draw [fill=qqqqff] (5.38,3.96) circle (1.5pt);
\draw [fill=qqqqff] (5.54,-0.14) circle (1.5pt);
\draw [fill=qqqqff] (0.76,-0.28) circle (1.5pt);
\draw [fill=qqqqff] (-0.28,2.18) circle (1.5pt);
\draw [fill=qqqqff] (6.26,2.34) circle (1.5pt);
\draw [fill=qqqqff] (2.92,3.22) circle (1.5pt);
\draw [fill=qqqqff] (3.02,0.62) circle (1.5pt);
\draw [fill=qqqqff] (2.4,1.64) circle (1.5pt);
\end{scriptsize}
\end{tikzpicture}

Let $A,B$, be non-empty sets. Let $Fun(A,B)$, be the set  of functions from $A$ to $B$, in other words, $Fun(A,B)= \{f \ | \ f: A \rightarrow B \}$.  If $B$ is a group, then we can turn  $Fun(A,B)$ into a group by defining a product,
$(fg)(a)=f(a)g(a), \ \  f, g\in Fun(A,B), \ \   a\in A,$
where the product on the right of the equation  is in $B$.
  If $f \in Fun([n], \mbox{Sym}([m]))$, then we define the mapping,  $$ A_f: V(\Gamma)\rightarrow V(\Gamma), \ \mbox{ by  this   rule},  \  A_f(x_1, \cdots ,x_n)=(f(1)(x_1), \cdots,  f(n)(x_n)).  $$ It is easy to show that the mapping
 $A_f$ is an automorphism of the Hamming graph $\Gamma =H(n,m)$, and hence the group,  $F=\{ A_f \ |  \ f\in  Fun([n], \mbox{Sym}([m])) \}$, is also a  subgroup
 of the Hamming graph $\Gamma =H(n,m)$.  Therefore, the subgroup which is  generated by $H$ and $F$  in the group $Aut(\Gamma)$, namely, $W=<H,F> $   is a subgroup of $Aut(\Gamma)$. In this paper,  we want to show that;  $$Aut(H(n,m))=W=<H,F>=\mbox{Sym}(\Omega) wr_I \mbox{Sym}([n])  $$\
There are various important families of graphs $\Gamma$,  in which we know that for a particular group $G$, we have
$G \leq Aut(\Gamma)$, but  showing  that in fact we have  $G = Aut(\Gamma)$, is a difficult task. For example note to the following cases. \newline

(1) \  The $Boolean\  lattice$  $BL_n, n \geq 1$, is the graph whose vertex set is the set of all subsets of $[n]= \{ 1,2,...,n \}$, where two subsets $x$ and $y$ are adjacent if and only if their symmetric difference has precisely one element. The $hypercube$  $Q_n$ is the graph whose vertex set is $ \{0,1  \}^n $, where two $n$-tuples  are adjacent if  they differ in precisely one coordinates. It is an easy task to show that $Q_n \cong BL_n $, and $ Q_n \cong Cay(\mathbb{Z}_{2}^n, S )$, where $\mathbb{Z}_{2}$ is
the cyclic group of order 2, and $S=\{ e_i \  | \  1\leq i \leq n \}, $ where  $e_i = (0, ..., 0, 1, 0, ..., 0)$,  with 1 at the $i$th position. It is an easy task to show that the set  $H= \{ f_\theta \ |\  \theta \in \mbox{Sym}([n]) \} $, $ f_\theta (\{x_1, ..., x_n \}) = \{ \theta (x_1), ..., \theta (x_n) \}$ is a subgroup of $Aut(BL_n)$, and hence $H$ is a subgroup of the group $Aut(Q_n)$.  We know that in every Cayley graph $\Gamma= Cay(G,S)$, the group $Aut(\Gamma)$ contains a subgroup isomorphic with the group $G$.  Therefore,  $\mathbb{Z}_{2}^n $ is a subgroup of $Aut(Q_n)$. Now,   showing that $Aut(Q_n) = <\mathbb{Z}_{2}^n, \mbox{Sym}([n])>( \cong \mathbb{Z}_{2}^n \rtimes \mbox{Sym}([n]))$, is  not an easy task [10]. \newline

(2) \  Let  $n,k \in \mathbb{ N}$ with $ k < \frac{n}{2}  $ and Let $[n]=\{1,...,n\}$.   The $Kneser\  graph$  $K(n,k)$ is defined as the graph whose vertex set is $V=\{v\mid v\subseteq [n], |v|=k\}$ and two vertices $v$ and $w$ are adjacent if and only if $|v\cap w|$=0.  The   Kneser graph $K(n,k)$  is  a vertex-transitive graph [6].  It is an easy task to show that the set  $H= \{ f_\theta \  |\  \theta \in \mbox{Sym}([n]) \} $,  $ f_\theta (\{x_1, ..., x_k \}) = \{ \theta (x_1), ..., \theta (x_k) \}$,  is a subgroup of  $ Aut ( K(n,k) )$ [6].   But,  showing  that
 $$ H= \{ f_\theta \ | \   \theta \in \mbox{Sym}([n]) \}= Aut ( K(n,k) )$$
  is not very easy [6 chapter 7, 13]. \newline

(3)\  Let  $n,k \in \mathbb{ N}$ with $ k < n,   $ and let $[n]=\{1,...,n\}$. The $Johnson\  graph$ $J(n,k)$ is defined as the graph whose vertex set is $V=\{v\mid v\subseteq [n], |v|=k\}$ and two vertices $v$ and $w  $ are adjacent if and only if $|v\cap w|=k-1$.   The Johnson\ graph $J(n,k)$ is a  vertex-transitive graph  [6].   It is an easy task to show that the set  $H= \{ f_\theta \ | \  \theta \in \mbox{Sym}([n]) \} $,  $f_\theta (\{x_1, ..., x_k \}) = \{ \theta (x_1), ..., \theta (x_k) \} $,    is a subgroup of $ Aut( J(n,k) ) $ [6].   It has been shown that   $Aut(J(n,k)) \cong \mbox{Sym}([n])$, if  $ n\neq 2k, $  and $Aut(J(n,k)) \cong \mbox{Sym}([n]) \times \mathbb{Z}_2$, if $ n=2k$,   where $\mathbb{Z}_2$ is the cyclic group of order 2 [3,7,12].
\section{Preliminaries}
In this paper, a graph $\Gamma=\Gamma(V,E) $ is considered as a simple undirected graph with vertex-set $V(\Gamma)=V $, and edge-set $ E(\Gamma)=E $.
For all
the terminology and notation not defined here, we follow [1,2,5,6]. \

The
group of all permutations of a set $V$ is denoted by $\mbox{Sym}(V)$  or
just $\mbox{Sym}(n)$ when $|V| =n $. A $permutation\  group$ $G$ on
$V$ is a subgroup of $\mbox{Sym}(V)$. In this case we say that $G$ act
on $V$. If $\Gamma$ is a graph with vertex set $V$, then we can view
each automorphism as a permutation of $V$, and so $Aut(\Gamma)$ is a
permutation group. Let $G$ act  on $V$, we say that $G$ is
$transitive$  (or $G$ acts $transitively$  on $V$), if there is just
one orbit. This means that given any two elements $u$ and $v$ of
$V$, there is an element $ \beta $ of  $G$ such that  $\beta (u)= v.
$

Let $\Gamma, \Lambda $ be  arbitrary graphs with vertex-set $V_{1},V_{2}$, respectively.   An isomorphism from  $\Gamma $ to $\Lambda $ is a bijection $\psi:V_{1}\longrightarrow V_{2}$ such that $\{x,y\}$ is an edge in $\Gamma $ if and only if  $\{\psi(x),\psi(y)\}$ is an edge in $\Lambda$. An isomorphism from a graph $\Gamma $ to itself is called an automorphism of the  graph $\Gamma$. The set of automorphisms of graph $\Gamma $ with the operation of composition of functions is a group, called the automorphism group of $\Gamma $ and denoted by $\mbox{Aut}(\Gamma)$. In most situations, it is difficult to determine the automorphism group of a graph, but there are various in the literature and some of the recent works appear  in the references [7,8,9,11,13,14,15,16,17].\

The graph $\Gamma $ is called $vertex$-$transitive,$  if $\mbox{Aut}(\Gamma) $ acts transitively on $V(\Gamma)$. In other words,  given any vertices $u,v $ of $\Gamma$, there is an $f\in \mbox{Aut}(\Gamma)$ such that $f(u)=v$.\

For $v\in V (\Gamma )$ and $G = Aut(\Gamma )$, the $stabilizer\  subgroup$  $G_{v}$ is the subgroup of $G$ containing of all automorphisms which fix $v$. In the vertex-transitive case all stabilizer subgroups $G_{v}$ are conjugate in $G$, and consequently isomorphic, in this case, the index of $G_{v}$ in $G$ is given by the equation,  $|G : G_{v}| = |G||G_{v}| = |V (\Gamma)|$. If each  stabilizer $G_{v}$ is the identity group, then every element of $G$, except the identity, does not fix any vertex  and we say that $G$ act semiregularly on $V$. We say that $G$ act regularly on $V$ if and only if
$G$ acts transitively and semiregularly on $V,$  and in this case we have $|V | = |G|. $ \

Let $N$ and $H$ be groups,  and let $  \phi: H  \rightarrow Aut(N)$ be a group homomorphism. In other words, the group $H$ acts on the group $N$, by this rule $n^h = \phi (h)(n)$, $n \in N, h \in H$. Note that in this case   we have $ {(n_1 n_2)}^h={n_1}^h {n_2}^h$, $n_1,n_2 \in N$.  The $semidirect \  product$  $N$ by $H$ which is  denoted by $N \rtimes H$ is a group on the set $N \times H$= $\{ (n,h) \ |  \  n\in N, h\in H  \}$, with the multiplication $(n,h)(n_1,h_1)=(n {(n_1)}^{-h}, hh_1) $. Note that the identity element of the group $ N \rtimes H$ is $(1_N, 1_H)$, and the inverse of the element $(n,h)$ is the element $ ({(n^{-1})}^{h}, {h}^{-1}) $.

\section{Main Results}
Let $\Gamma $ be a connected graph with diameter $d$. Then we can partition the vertex-set  $V(\Gamma) $ with respect to the distances of vertices from a fixed vertex. Let $v$ be a fixed vertex of the graph $\Gamma $. We denote the set of vertices at distance $i$ from $v$, by $\Gamma _{i}(v)$. Thus it is obvious that $\{v\}=\Gamma _{0}(v)$ and $ \Gamma_{1}(v)  =  N(v)   $, the set of adjacent vertices to vertex $v$, and $V (\Gamma )$ is partitioned into the disjoint subsets $\Gamma_{0}(x), ..., \Gamma_{D}(x)$. If $\Gamma =H(n,m)$,  then it is clear that two vertices are at distance $k$ if and only if they differ in exactly $k$ coordinates. Then the maximum distance occurs when the two vertices
(regarded as ordered $n$-tuples) differ in all $n$ coordinates. Thus the diameter of $H(n,m)$ is equal to $n$.\\

\begin{lem}
Let $m \geq 3$ and   $\Gamma =H(n,m)$. Let $x\in V(\Gamma )$, $\Gamma_{i}=\Gamma_{i}(x)$ and $v\in \Gamma_{i}$. Then we have;
\begin{center}
$\displaystyle \bigcap _{w\in \Gamma_{i-1}\cap N(v)}(N(w)\cap \Gamma _{i})=\{v\} $.
\end{center}
\end{lem}
\begin{proof}
It is obvious that
 \begin{center}
$v\in \displaystyle \bigcap _{w\in \Gamma_{i-1}\cap N(v)}(N(w)\cap \Gamma _{i}) $.
\end{center}
Let $x=(x_{1},\cdots ,x_{n})$.  Since the Hamming graph $H(n,m)$  is a  distance-transitive graph [3], then we can assume that, $v=(x_{1},\cdots ,x_{n-i},y_{n-i+1},\cdots ,y_{n})$, where $y_{j}\in \mathbb{Z}_{m}-\{x_{j}\}$ for all $j=n-i+1,\cdots ,n$.

Let $w\in \Gamma_{i-1}\cap N(v)$. Then $w,x$ differ in exactly $i-1$ coordinates and $w,v$ differ in exactly one coordinate.  Note that,  if in  $v$   we change one of $x_{j}$s,  where $j=1,\cdots, n-i$, then we obtain a vertex $u$ such that  $d(u,x)\geq i+1$. Thus,  $w$ has a form such as;
\begin{center}
$w=w_{r}=(x_{1},\cdots ,x_{n-i},y_{n-i+1},\cdots,y_{r-1},x_{r},y_{r+1},\cdots ,y_{n})$

\end{center}
 We show that if $u\in \Gamma_{i}$ and $u\neq v$ and $u$ is adjacent to some $w_{r}$, then there is some $w_{p}$ such that $u$ is not adjacent to $w_{p}. $\\
If $v\neq u\in \Gamma_{i}$ is adjacent to $w_{r}$ then $u$ has one of the following forms;\newline
\item[(i)] $u_{1}=(x_{1},\cdots ,x_{n-i}, y_{n-i+1},\cdots ,y_{r-1},y,y_{r+1},\cdots ,y_{n})$, where $y \in \mathbb{Z}_{m} $, and  $y\neq y_{r},x_{r}$ (note that since $m \geq 3$, hence there is  such a  $ y$). \newline
\item[(ii)] $u_{2}=(x_{1}\cdots,x_{j-1},y,x_{j+1},\cdots,x_{n-i},y_{n-i+1},\cdots,y_{r-1},x_{r},y_{r+1},\cdots,y_{n})$,\newline  where $y \in \mathbb{Z}_{m} $, and $y\neq x_{j}$.\newline
In the case (i), $u_{1}$ is not adjacent to $w_{t}$, for all possible $t$,  $t\neq r$.\\
In the case (ii), it is obvious that $u_{2}$ is also  not adjacent to $w_{t}$ for all possible $t$, $t\neq r$.\\
Our argument shows that if $ u\in \Gamma_{i}$, and $u\neq v$,  then there is some $w_{r}$ such that $u$ is not adjacent to $w_{r}$, in other words $u \notin N(w_r)$. Thus we have;
\begin{center}

$\displaystyle \bigcap _{w\in N(v)\cap \Gamma_{i-1}}(N(w)\cap \Gamma _{i})=\{v\} $.
\end{center}

\end{proof}
Let $I=\{\gamma_1, ... ,\gamma_{n}\}$ be a set  and $K$ be a group.  Let $Fun(I, K)$  be the set of all functions from $I$ into $K$. We can turn  $Fun(I, K)$ into a group by defining a product:
$$(fg)(\gamma)=f(\gamma)g(\gamma), \ \  f, g\in Fun(I, K), \ \   \gamma\in I,$$
where the product on the right of the equation  is in $K$. Since $I$ is finite, the group  $Fun(I, K)$ is isomorphic to $K^{n}$ (the direct product of $n$ copies of $K$),  by the isomorphism $f\mapsto (f(\gamma_1), ... , f(\gamma_{n}))$. Let $H$   be a group and  assume  that  $H$ acts on the nonempty set $I$. Then,  the wreath product of  $K$ by $H$ with respect to this action is  the semidirect product $Fun(I, K)\rtimes H$   where $H$ acts on the group $Fun(I, K)$,  by the following rule,
 $$ f^{x}(\gamma)=f(\gamma^{ x^{-1}}),\    f\in Fun(I, K),  \gamma \in I,    \, x \in  H.$$
 We denote this group by $K wr_{I}H$. Consider the wreath product $G=K wr_{I}H$. If $K$ acts on a set $\Delta$ then
 we can define an action of $G$ on $\Delta\times I$ by the following rule,
 $$(\delta, \gamma)^{(f, h)}=(\delta^{f(\gamma)}, \gamma^h), \ \ (\delta, \gamma)\in \Delta\times I,$$
  where $(f, h)\in Fun(I, K)\rtimes H=K wr_{I}H$. It is clear that if $I, K$ and H, are finite sets, then  $G=K wr_{I}H$, is a finite group, and we have $|G|= {|K|}^{|I|} |H|$. \\

We have the following theorem [4].

\begin{thm}
Let $\Gamma $ be a graph with $n$ connected components  $\Gamma_{1},\Gamma_{2},\cdots,\Gamma_{n}$, where $\Gamma_{i}$ is isomorphic to $\Gamma_1$ for all $i \in [n]= \{1,\cdots,n \}=I$. Then  we have,  $Aut(\Gamma)=Aut(\Gamma_1)wr_{I}\mbox{Sym}([n])$.
\end{thm}

\begin{lem}
Let  $n \geq 2,\  m \geq 3$. Let $v$  be a vertex of the Hamming graph $H(n,m)$.  Then,  $\Gamma_{1} =< N(v) >$, the induced subgraph of $N(v)$ in $H(n,m)$, is isomorphic with $nK_{m-1}$, where $nK_{m-1}$ is the disjoint union of $n$ copies of the  complete graph $K_{m-1}$.
\end{lem}
\begin{proof}
Let $v=(v_{1},\cdots,v_{n})$. Then,  for all $i$,  $i=1,\cdots,n$,  there are $m-1$ elements $w_{j}$,  $w_{j}\in \mathbb{Z}_{m}-\{v_{i}\}$. Let $x_{ij}=(v_{1},\cdots,v_{i-1},w_{j},v_{i+1},\cdots,v_{n})$, $1\leq i\leq n, 1\leq j\leq n-1$. Then,  $N(v)=\{x_{ij}:\;\; 1\leq i\leq n,\; 1\leq j\leq m-1\}$. Let $x_{ij},x_{rs}$,  be two vertices in $\Gamma_{1}=<N(v)>$, then $x_{ij}, x_{rs}$ are adjacent in $\Gamma_{1}$ if and only if $i=r$. Note that two vertices,  $(v_{1}, ..., v_{i-1}, w_{j} , v_{i+1}, ..., v_{n})$ and $(v_{1}, ..., v_{i-1}, w_{s} , v_{i+1}, ..., v_{n})$ differ in only one coordinate.
Therefore,  for each $i=1,\cdots ,n$, there are $m-2$ vertices  $w_{ir}$ in $ \Gamma_{1} $ which are adjacent  to the vertex  $w_{ij}$,  where  $r\neq j$. Now, it is obvious that the  subgraph induced by the set $ \{ x_{ij}: 1 \leq j \leq m-1  \} $, is isomorphic with $K_{m-1}$, the complete graph of order $m-1$.  Now, it is easy to see that, the subgraph induced by the set  $ \{x_{ij}:\  i=1,\cdots,n,\  j=1,\cdots,m-1\}$, is isomorphic with $ nK_{m-1}$, the disjoint union of $n$ copies of the  complete graph $K_{m-1}$.

\end{proof}
We now are  ready to prove the main result of this paper.
\begin{thm}
Let $n \geq 2,\  m \geq 3$,  and  $\Gamma=H(n,m)$  be a  Hamming graph. Then Aut$(\Gamma)\cong \mbox{Sym}([n])wr_I \mbox{Sym}([m]) $, where $I= [n]= \{1,2,\cdots n \}$.
\end{thm}
\begin{proof}
Let $G = \mbox{Aut}(\Gamma)$. Let $x \in V = V (\Gamma)$, and $G_{x} = \{f \in G\; |\;f(x) = x\}$ be the stabilizer
subgroup of the vertex $x$ in Aut$(\Gamma)$. Let $< N(x) >= \Gamma_{1}$ be the induced subgroup of $N(x)$ in $\Gamma$. If $f\in G_{x}$ then $f_{|N(x)}$, the restriction of $f$ to $N(x)$ is an automorphism of the graph  $\Gamma_{1}$. We
define the mapping $\psi : G_{x}\longrightarrow \mbox{Aut}(\Gamma_{1})$ by this rule,  $\psi(f) = f_{|N(x)}$.  It is an easy task to show that $\psi$ is
a group homomorphism. We show that $Ker(\psi)$  is the identity group. If $f \in Ker(\psi)$, then $f(x) = x$ and $f(w) = w$ for every $w \in N(x)$.   Let $\Gamma_{i}$ be the set of vertices of $\Gamma$  which are at distance $i$ from the vertex $x$. Since,  the diameter of the graph $\Gamma = H(n,m)$, is $n$,   then $V = V (\Gamma) = \displaystyle \bigcup_{i=0}^{n}\Gamma_{i}$.
We prove by induction on $i$, that $f(u) = u$ for every $u \in \Gamma_{i}$. Let $d(u, x)$ be the distance of the vertex $u$ from $x$. If $d(u, x) = 1$, then $u \in \Gamma_{1}$ and we have $f(u) = u$. Assume that $f(u) = u$,
when $d(u, x) = i-1$. If $d(u, x) = i$, then by   Lemma 1.    $\{u\} =\displaystyle \bigcap_{w\in \Gamma_{i-1}\cap N(u)}(N(w)\cap \Gamma_{i})$.
Note that if $w\in \Gamma_{i-1}$,  then $d(w, x) = i - 1$, and hence  $f(w) = w$. Therefore,
\begin{center}
$\{f(u)\} = \displaystyle \bigcap_{w\in \Gamma_{i-1}\cap N(u)}(N(f(w))\cap \Gamma_{i})=\displaystyle \bigcap_{w\in \Gamma_{i-1}\cap N(u)}(N(w)\cap \Gamma_{i})= u$.
\end{center}
Thus,  $f(u)=u$ for all $u\in V(\Gamma)$,  hence  we have $Ker(\psi)=\{1\} $. On the other hand,
\begin{center}
$\frac{G_{v}}{Ker(\psi)}\cong \psi(G_{v})\leq \mbox{Aut}(\Gamma_{1}) $, hence  $G_{v} \cong \psi (G_{v})\leq \mbox{Aut}(\Gamma_{1})$.
\end{center}
Thus,  $|G_{v}|\leq | \mbox{Aut}(\Gamma_{1})|$.\\
 We know by Lemma 3.  that $\Gamma_{1}\cong nK_{m-1}$. We know that,  Aut$(K_{m-1})\cong \mbox{Sym}([m-1]) $.   Then,  by the above equation,  we have;
\begin{center}
$|G_{v}|\leq |\mbox{Aut}(\Gamma_{1})|=|\mbox{Sym}([m-1])wr_{I}\mbox{Sym}([n])|=((m-1)!)^{n}n!$,
\end{center}
 where
 $I= [n]=\{1, \cdots,  n  \}$.

Since $\Gamma = H(n,m)$ is a vertex-transitive graph, then  we have $|V (\Gamma)| = |G|
|G_{v}|$, and therefore;  $$|G| =|G_{v}||V (\Gamma)| \leq |\mbox{Aut}(nK_{m-1})|m^{n}=m^{n}((m-1)!)^{n}n!=(m!)^{n}n! \ \ \ \ \ (*)$$ \
We have seen (in the introduction section of this paper) that if $\theta \in \mbox{Sym}([n]$), where $ \Omega=[n]=\{1,\cdots,  n  \} $,   then
$$ f_\theta : V( H(n,m) )\longrightarrow V( H(n,m)),
 f_\theta (x_1, ..., x_n ) =   (x_{\theta (1)}, ..., x_{\theta (n)}),  $$
 is an automorphism of the Hamming graph $ H(n,m),  $ and the mapping  \
  $ \psi : \mbox{Sym} ([n]) \longrightarrow Aut ( H(n,m) )$, defined by
 this  rule,  $ \psi ( \theta ) = f_\theta $,  is an injection. Therefore, the set  $H= \{ f_\theta \  |\  \theta \in \mbox{Sym}([n]) \} $,  is a subgroup of $ Aut (( H(n,m) )) $,  which is isomorphic with $\mbox{Sym}([n])$.  Hence,  we have  $\mbox{Sym}([n]) \leq Aut ( H(n,m) )$. \

 On the other hand,  if $f \in Fun([n], \mbox{Sym}([m]))$, then we define the mapping;  $A_f: V(\Gamma)\rightarrow V(\Gamma),$    by    this   rule,  \newline   $  A_f(x_1, \cdots ,x_n)$=$(f(1)(x_1), \cdots,  f(n)(x_n)).$    It is an easy task  to show that the mapping
 $A_f$ is an automorphism of the Hamming graph $\Gamma$, and hence the group,  $F=\{ A_f \ |  \ f\in  Fun([n], ([m])) \}$, is a subgroup
 of the Hamming graph $\Gamma =H(n,m)$.  Therefore, the subgroup which is  generated by $H$ and $F$   is  in the group $Aut(\Gamma)$, namely, $W=<H,F>$    is a subgroup of $Aut(\Gamma)$. Note that $W= \mbox{Sym}([m]) wr_I \mbox{Sym}([n])$, where $I=[n]=\{1,2,\cdots,n  \}$. Since, the subgroup $W$ has  $(m!)^{n}n!$ elements, then by $(*)$, we conclude that;  $$Aut(\Gamma)=W= \mbox{Sym}([m]) wr_I \mbox{Sym}([n])$$

\end{proof}


\end{document}